\documentclass[reqno,12pt]{amsart}

\usepackage[margin=2.5cm]{geometry}
\usepackage[T1]{fontenc}
\usepackage[utf8]{inputenc}
\usepackage[english]{babel}
\usepackage{amsmath,amssymb,amsthm,amsrefs}
\usepackage[normalem]{ulem}
\usepackage[adobe-utopia]{mathdesign}
\usepackage{enumitem}
\usepackage{xcolor}
\usepackage{comment}
\usepackage{tikz}
\usepackage{centernot}
\usepackage{mathtools}
\usepackage{hyperref}

\newcommand{\D}{\mathbb{D}}
\newcommand{\T}{\mathbb{T}}
\newcommand{\TT}{\mathcal{T}}
\newcommand{\norm}[1]{\left\lVert #1\right\rVert}
\newcommand{\one}{\emph{one-component}}
\newcommand{\Ret}{\mathcal{R}}
\DeclareMathOperator{\arctanh}{arctanh}
\numberwithin{equation}{section}
\newtheorem{theorem}{Theorem}[section]
\newtheorem{proposition}[theorem]{Proposition}
\newtheorem{lemma}[theorem]{Lemma}
\newtheorem{corollary}[theorem]{Corollary}
\theoremstyle{definition}
\newtheorem{definition}[theorem]{Definition}
\theoremstyle{remark}

\title[Proof of Baranov--Bessonov--Kapustin conjecture]
{A proof of the Baranov--Bessonov--Kapustin conjecture}
\date{\today}
\author{Carlo Bellavita}
\address{Dipartimento di Matematica ``F. Enriques'', Dipartimento di
Eccellenza MUR 2023--2027, Universit\`a degli Studi di Milano, Via
C.\ Saldini 50, 20133 Milano, Italy}
\email{carlo.bellavita@gmail.com}
\email{carlo.bellavita@unimi.it}

\author{Jani A. Virtanen}
\address{University of Eastern Finland, Joensuu, Finland. University of Reading, Reading, UK. University of Helsinki, Helsinki, Finland}
\email{jani.virtanen@uef.fi}

\thanks{
The first author is a member of the Gruppo Nazionale per l’Analisi Matematica,
la Probabilità e le loro Applicazioni (GNAMPA) of the Istituto Nazionale di Alta
Matematica (INdAM). The second author was supported in part by the Engineering and Physical Sciences Research Council grant EP/Y008375/1 and the Research Council of Finland grant 375631.}

\begin{document}
\begin{abstract}
We prove the conjecture of Baranov, Bessonov and Kapustin: \emph{for a
non-constant inner function $\theta$, every bounded truncated Toeplitz
operator on the model space $K_\theta$ has a bounded symbol if and only if
$\theta$ is one-component}.
\end{abstract}
\maketitle

\section{Introduction}

%\vspace{11 pt}
\noindent Let $\mathbb{D}$ be the open unit disk, let $\mathbb{T}=\partial\mathbb{D}$,
and let $m$ be normalized Lebesgue measure on $\mathbb{T}$. We identify
functions in the Hardy space $H^2$ with their almost everywhere defined
boundary values and use the inner product
\begin{equation*}
\langle f,g\rangle=\int_{\mathbb{T}} f\overline{g}\,dm.
\end{equation*}
For a non-constant inner function
$\theta$, consider the model space
\[
K_\theta=H^2\ominus\theta H^2
\]
\cite{Cima, GR, Rossmodel}, and its bounded part
\[
K_\theta^\infty=K_\theta\cap H^\infty .
\]
We denote by $P_\theta$ the orthogonal projection from $L^2(m)$ onto
$K_\theta$.

\vspace{11 pt}\noindent
Given $\varphi\in L^2(m)$, the \emph{truncated Toeplitz
operator} with symbol $\varphi$ is initially defined on the dense
subspace $K_\theta^\infty$ by
\[
A_\varphi^\theta x=P_\theta(\varphi x).
\]
We write $\TT(\theta)$ for the space of these operators that extend
boundedly to $K_\theta$. Sarason developed their systematic theory
in~\cite{sarason}. See also\cite{Chal-fri, Fricain} as excellent surveys. When $\theta(z)=z^n$, this construction gives precisely
the $n\times n$ Toeplitz matrices in the basis $1,z,\ldots,z^{n-1}$.
Clearly for $\varphi\in L^\infty(m)$, the operator $A^\theta_\varphi$ is bounded and
\begin{equation}\label{eq:contractive-symbol}
\norm{A_\varphi^\theta}\le\norm{\varphi}_\infty.
\end{equation}

\noindent
A distinctive feature of truncated Toeplitz operators is the
nonuniqueness of their symbols. The zero-symbol characterization
in~\cite[Theorem 3.1]{sarason} states that
\[
A_\varphi^\theta=0
\quad\Longleftrightarrow\quad
\varphi\in\theta H^2+\overline{\theta H^2},
\qquad \varphi\in L^2(m),
\]
so two $L^2$ symbols define the same operator precisely when their
difference belongs to this subspace.
Thus an unbounded symbol $\psi$ may represent an operator $A^\theta_\psi$ that also has a
bounded symbol. The question raised by Sarason is whether  \emph{every}
operator $A^\theta_\psi \in \TT(\theta)$ admits at least one such bounded representative $\varphi$~\cite{sarason}.
To express this question quantitatively, define
\begin{equation*}\label{eq:beta}
\beta_\theta(A)=\inf\bigl\{\norm{\varphi}_\infty:
 \varphi\in L^\infty(m),\ A_\varphi^\theta=A\bigr\},
\end{equation*}
with the convention $\inf\varnothing=+\infty$ and
{
\[
B_\theta = \sup\{ \beta_\theta(A) : A\in\mathcal T(\theta),\ \|A\|\le 1\}.
\]
}
The \emph{bounded symbol
property} is exactly the finiteness of $B_\theta$.
%on all of $\TT(\theta)$.

\vspace{11 pt}\noindent 
Some partial answers are classical. If $A_\varphi^\theta$ is bounded and
$\varphi\in H^2$, the commutant lifting theorem produces a symbol in
$H^\infty$~\cite[Section~4]{sarason}. Moreover
Rochberg~\cite{Rochberg} obtained for Wiener--Hopf operators on Paley--Wiener
spaces that every bounded truncated Toeplitz operator has a bounded symbol. In general,
however, the answer to the original question of Sarason is negative: in~\cite{Bar-Chal} the authors found inner functions
$\theta$ for which $\TT(\theta)$ contains operators, even of rank one, that
have no bounded symbol. 

\vspace{11 pt}\noindent
This leads to the natural problem of describing the
model spaces $K_\theta$ for which every bounded truncated Toeplitz operator
has a bounded symbol. As we shall see, the answer is expressed by the
geometry of the sublevel sets of $\theta$.

\begin{definition}
A non-constant inner function $\theta$ is \one{} if, for some $R\in(0,1)$,
the set $\{z\in\D:|\theta(z)|<R\}$ is nonempty and connected.
\end{definition}
\noindent This class originates in Cohn's work on Carleson measures for model
spaces~\cite{cohn}. Further characterizations were obtained by
Aleksandrov in his seminal articles about the embeddings of model spaces~\cite{Ale1, Ale2} and by several other authors in subsequent articles~\cite{Vol, B2005, Bess, Cim-Mort1, Rei, Nic-Rei, Bell-Nic, bell-mash, Bell-Dell}.
Baranov, Bessonov, and Kapustin proved that the
one-component condition implies the bounded symbol property and
conjectured that it is also necessary~\cite[Corollary 2.5 and
Conjecture 2.6]{bbk}. We will prove this necessity in this article. 

\begin{theorem}\label{thm:main}
Let $\theta$ be a non-constant inner function. The following conditions
are equivalent:
\begin{enumerate}[label=\textup{(\roman*)}]
\item Every operator in $\TT(\theta)$ has a symbol in $L^\infty(m)$.
\item The normalized rank-one operators have uniformly bounded symbol
\[
M_\theta=\sup_{\lambda\in\D}\beta_\theta(R_\lambda^\theta)<\infty.
\]
\item The function $\theta$ is one-component.
\end{enumerate}

\end{theorem}
We recall that the rank-one Toeplitz operator $R_\lambda^\theta$ is defined as
\[
R_\lambda^\theta=
\frac{\widetilde k_\lambda^\theta\otimes k_\lambda^\theta}
 {\|k^\theta_\lambda\|^2},
\]
where $k^\theta_\lambda$ is the reproducing kernel of $K_\theta$, $\widetilde k_\lambda^\theta$ is the conjugate kernel and $(v\otimes w)x=\langle x,w\rangle v$.

\vspace{11 pt}\noindent 
Truncated Toeplitz operators form a natural class of operators on model
spaces, and the Baranov--Bessonov-Kapustin conjecture\footnote{{
A word on terminology. We refer to this statement as the Baranov--Bessonov--Kapustin conjecture, following its explicit
formulation in \cite[p.~3444]{bbk}, where we first encountered it.
Several colleagues have pointed out that it may already have appeared in \cite{Bar-Chal}; we have not found an explicit statement there, but we cannot rule out an earlier formulation. The name is adopted solely for convenience and is not intended to attribute priority.} } studied here is of particular interest in view of
its several equivalent reformulations in \cite[Theorem 2.4]{bbk} and \cite[Theorem 3.2]{O'Loughlin}.

\vspace{11 pt}\noindent 
The rest of the paper is organized as follows. Section~\ref{Sect2} collects
some preliminary material. The key lemmas for the proof of
Theorem~\ref{thm:main} are established in Sections~\ref{Sect3}
and~\ref{Sect4}. The proof is completed in Section~\ref{Sect5}, together with some immediate corollaries. Section~\ref{Sect6} discusses the
consequences of Theorem~\ref{thm:main}.

\vspace{11 pt}
\section{Preliminaries: model spaces, rank-one symbols, and quotient norms}\label{Sect2}
\noindent 
In the model space $K_\theta$, it is defined the canonical conjugation~\cite{Gar-Put1, Gar-Put2}
\[
C_\theta h=\theta\bar z\overline h,\qquad h\in K_\theta.
\]
It satisfies
\begin{equation}\label{eq:conjugation}
C_\theta^2=I,
\qquad \langle C_\theta f,C_\theta g\rangle=\langle g,f\rangle,
\qquad |C_\theta f|=|f|\quad\text{a.e. on }\T.
\end{equation}
In particular, it preserves $K_\theta^\infty$.
A direct boundary calculation gives
\begin{equation*}\label{eq:conjugate-kernel}
C_\theta k_\lambda^\theta
=\theta\bar z\,
  \frac{1-\theta(\lambda)\bar\theta}{1-\lambda\bar z}
=\frac{\theta-\theta(\lambda)}{z-\lambda}
=\widetilde k_\lambda^\theta. 
\end{equation*}
Both kernels $k^\theta_\lambda$, $\widetilde k_\lambda^\theta$ belong to $K_\theta^\infty$ when $\lambda \in \D$.
First of all, we look for the bounded symbol of $R_\lambda^\theta$~\cite[Corollary 5.1]{sarason}.
\begin{lemma}\label{lem:rank-one}
For every $\lambda\in\D$, the function
\begin{equation}\label{eq:rank-one-symbol}
\varphi_\lambda(\zeta)=
\frac{\theta(\zeta)}{\|k_\lambda^\theta\|^2(\zeta-\lambda)},
\qquad \zeta\in\T,
\end{equation}
is a bounded symbol for
\[
R_\lambda^\theta=
\frac{\widetilde k_\lambda^\theta\otimes k_\lambda^\theta}
 {\|k^\theta_\lambda\|^2}.
\]
Moreover,
\begin{equation*}\label{eq:rank-one-norms}
\norm{R_\lambda^\theta}
=\norm{R_\lambda^\theta}_{S_1}=1,
\qquad
1\le\beta_\theta(R_\lambda^\theta)<\infty.
\end{equation*}
\end{lemma}
\begin{proof}
The bound $|\zeta-\lambda|\ge1-|\lambda|$ makes $\varphi_\lambda$
%\eqref{eq:rank-one-symbol}
a bounded function. For
$x,y\in K_\theta^\infty$, we have that 
\begin{align}
\langle A_{\varphi_\lambda}^\theta x,C_\theta y\rangle
&=\frac{1}{\|k_\lambda^\theta\|^2}
  \int_\T\frac{\zeta x(\zeta)y(\zeta)}{\zeta-\lambda}\,dm(\zeta)
  =\frac{x(\lambda)y(\lambda)}{\|k_\lambda^\theta\|^2}.
\label{eq:cauchy-rank-one}
\end{align}
On the other hand, \eqref{eq:conjugation} gives
\[
\langle\widetilde k_\lambda^\theta,C_\theta y\rangle
=\langle y,k_\lambda^\theta\rangle=y(\lambda),
\]
so the last expression in \eqref{eq:cauchy-rank-one} is exactly
$\langle R_\lambda^\theta x,C_\theta y\rangle$.
Density of $K_\theta^\infty$, together with its invariance under
$C_\theta$, identifies the two bounded operators.

A rank-one operator $v\otimes w$ has one nonzero singular value,
$\norm{v}_2\norm{w}_2$. Since $C_\theta$ is isometric, this value for
$R_\lambda^\theta$ is
$1$.
Finally, \eqref{eq:contractive-symbol} and the fact that $\|R^\theta_\lambda\|=1$ imply
$\beta_\theta(R_\lambda^\theta)\ge1$.

\end{proof}
\noindent The
following lemma shows that, whenever every bounded truncated Toeplitz
operator has a bounded symbol, {then $B_\theta$ is finite and } one can choose the symbol with norm
controlled by $\norm{A}$~\cite[Section 6]{bbk}.
\begin{lemma}\label{lem:quotient}
The space $\TT(\theta)$ is closed in operator norm. If every element of
$\TT(\theta)$ has a bounded symbol, then $B_\theta<\infty$
and
\begin{equation}\label{eq:uniform-symbol-bound}
\beta_\theta(A)\le B_\theta\norm{A}
\qquad(A\in\TT(\theta)).
\end{equation}
\end{lemma}
\begin{proof}
Sarason's characterization~\cite[Theorem 8.1]{sarason} states
that a bounded operator $A$ on $K_\theta$ belongs to $\TT(\theta)$ if and
only if
\begin{equation*}\label{eq:sarason}
\langle Af,f\rangle=\langle A(zf),zf\rangle
\quad\text{whenever } f\in K_\theta\text{ and }zf\in K_\theta.
\end{equation*}
For each such $f$, this identity is preserved by operator-norm limits.
Thus $\TT(\theta)$ is a closed linear subspace of the bounded operators
on $K_\theta$, and is a Banach space.
The continuous linear map
\[
Q_\theta:L^\infty(m)\longrightarrow\TT(\theta),
\qquad Q_\theta\varphi=A_\varphi^\theta
\]
is contractive by \eqref{eq:contractive-symbol}. Its kernel
$N_\theta$ is closed, so the quotient
$E_\theta=L^\infty(m)/N_\theta$ is Banach. The induced map
\[
\widetilde Q_\theta:E_\theta\longrightarrow\TT(\theta)
\]
is continuous and injective. If every bounded truncated Toeplitz
operator has a bounded symbol, it is also surjective. The bounded inverse
theorem gives
\[
\norm{\widetilde Q_\theta^{-1}A}_{E_\theta}
\le B_\theta\norm{A}.
\]
If $\varphi$ is any bounded symbol of $A$, its entire symbol fiber is
$\varphi+N_\theta$. Therefore
\[
\norm{\widetilde Q_\theta^{-1}A}_{E_\theta}
=\inf_{\nu\in N_\theta}\norm{\varphi+\nu}_\infty
=\beta_\theta(A),
\]
which proves \eqref{eq:uniform-symbol-bound}.

\end{proof}

%\vspace{11 pt}
\section{The central estimate by orthogonal projection}\label{Sect3}

\noindent We start with a preliminary estimate.

\begin{proposition}\label{prop:center}
Let $u$ be a non-constant inner function.
Then
\begin{align}
1-\frac{|u'(0)|^2}{(1-|u(0)|^2)^2}&\le\beta_u(R_0^u)^2(1-|u(0)|^2),
\label{eq:center-l2}\\
1-\frac{|u'(0)|}{(1-|u(0)|^2)}&\le\beta_u(R_0^u)\norm{k^u_0}_1.
\label{eq:center-l1}
\end{align}
\end{proposition}
\begin{proof}
Define
\begin{equation*}\label{eq:test-function}
f=k^u_0-\frac{\overline{u'(0)}}{1-|u(0)|^2} \widetilde k_0^u \in K_u^\infty.
\end{equation*}
We notice that
\begin{equation}\label{eq:projection-data}
\norm{f}_2^2
=\|k^u_0\|^2-\frac{|u'(0)|^2}{(1-|u(0)|^2)^2}\|C_uk^u_0\|^2
=(1-|u(0)|^2)\left(1- \frac{|u'(0)|^2}{(1-|u(0)|^2)^2}\right)
=f(0).
\end{equation}
Involution of $C_u$ gives
\[
C_uf=C_uk^u_0-\frac{{u'(0)}}{1-|u(0)|^2}k^u_0= \widetilde k_0^u -\frac{{u'(0)}}{1-|u(0)|^2}k^u_0,\qquad (C_uf)(0)=0.
\]
Consequently, $(C_uf)/z$ belongs to $H^2$. Its boundary values satisfy
\[
\frac{C_uf}{z}=u\bar z\,\overline{zf}.
\]
Since $zf\in H^2$, the characterization of model space elements~\cite[Proposition 5.4]{Rossmodel} implies that
$zf\in K_u$. Boundedness follows from the fact that $f\in H^\infty$.

\vspace{11 pt}\noindent 
Let $\psi\in L^\infty(m)$ be any symbol of $R_0^u$: at least one exists
by Lemma~\ref{lem:rank-one}. For any $x,y\in K_u^\infty$, the same
pairing as in \eqref{eq:cauchy-rank-one} yields
\begin{equation}\label{eq:bilinear}
\int_\T\psi z\bar u\,xy\,dm
=\langle R_0^u x,C_uy\rangle
=\frac{x(0)y(0)}{1-|u(0)|^2}.
\end{equation}
The definitions of $k^u_0$ and $\widetilde k^u_0$ give the pointwise analytic identities
\[
k^u_0+\bar u(0) z\widetilde k^u_0=(1-\bar u(0) u)+\bar u(0)(u-u(0))=1-|u(0)|^2.
\]
\begin{equation}\label{eq:two-products}
(1-|u(0)|^2)f=k^u_0f+\bar u(0)\,\widetilde k^u_0 zf.
\end{equation}
All four factors $k^u_0,f,\widetilde k^u_0,zf$ belong to $K_u^\infty$. Hence
\eqref{eq:bilinear} applies to both products in
\eqref{eq:two-products}, and gives
\begin{align}
(1-|u(0)|^2)\int_\T\psi z\bar u f\,dm
&=\frac{k^u_0(0)f(0)}{1-|u(0)|^2}
  +\bar u(0)\frac{\widetilde k^u_0(0)(zf)(0)}{1-|u(0)|^2}
 =f(0).
\label{eq:test-integral}
\end{align}

\noindent 
Let $\eta=|u'(0)|/(1-|u(0)|^2)$ and  put $h=1-\eta^2\ge0$, because of Schwarz--Pick lemma~\cite{Rudin}.
Equations \eqref{eq:projection-data} and \eqref{eq:test-integral} imply
\begin{equation}\label{eq:basic-bound}
h=1-\frac{|u'(0)|^2}{(1-|u(0)|^2)^2}=\frac{|f(0)|}{1-|u(0)|^2}=\left|\int_\T\psi z\bar u f\,dm\right|
\le\norm{\psi}_\infty\norm{f}_1.
\end{equation}
Because $m(\T)=1$, Cauchy--Schwarz and
\eqref{eq:projection-data} give
\[
%1-\frac{|u'(0)|^2}{(1-|u(0)|^2)^2}=
h\le\norm{\psi}_\infty\norm{f}_2
=\norm{\psi}_\infty\sqrt{(1-|u(0)|^2) h}.
\]
If $h=0$, \eqref{eq:center-l2} is immediate. If $h>0$, division by
$\sqrt h$ and squaring give $h\le\norm{\psi}_\infty^2(1-|u(0)|^2)$.
Taking the infimum over all bounded symbols proves
\eqref{eq:center-l2}.

For the second estimate, the triangle inequality yields
\[
\norm{f}_1
\le\norm{k^u_0}_1+\frac{|u'(0)|}{1-|u(0)|^2}\norm{\widetilde k^u_0}_1
=(1+\eta)\norm{k^u_0}_1.
\]
Using this in \eqref{eq:basic-bound} and dividing by $1+\eta>0$ gives
$1-\eta\le\norm{\psi}_\infty\norm{k^u_0}_1$. The infimum over $\psi$
proves \eqref{eq:center-l1}.

\end{proof}

\noindent 
The problem of Proposition \ref{prop:center} is that it provides estimates only at $0$, i.e, for $R^u_0$. We have to extend them to every point $\lambda \in \D$.
The following proposition is the first key ingredient of the proof of Theorem \ref{thm:main}. 
The weighted composition below is the standard precomposition
transformation for truncated Toeplitz operators; compare it with
\cite[Proposition 4.1]{cgrw}.  

\begin{proposition}\label{prop:covariance}
For every non-constant inner function $\theta$ and every $\lambda\in\D$,
\begin{align}
1-\eta_\theta(\lambda)^2
&\le\beta_\theta(R_\lambda^\theta)^2
             (1-|\theta(\lambda)|^2),
\label{eq:global-l2}
%\\
%1-\eta_\theta(\lambda)
%&\le\beta_\theta(R_\lambda^\theta)
%             \Phi(|\theta(\lambda)|),
%\label{eq:global-l1}
%\notag
\end{align}
where 
\[
\eta_\theta(\lambda)=\frac{(1-|\lambda|^2)|\theta'(\lambda)|}{(1-|\theta(\lambda)|^2)}.
\]
\end{proposition}
\begin{proof}
Fix $\lambda \in \D$. Define
\begin{equation*}\label{eq:automorphism}
\psi_\lambda(z)=\frac{\lambda+z}{1+\bar\lambda z},
\quad s_\lambda(z)=\frac{\sqrt{1-|\lambda|^2}}{1+\bar\lambda z},
\quad Uh=s_\lambda(h\circ\psi_\lambda),
\quad u=\theta\circ\psi_\lambda.
\end{equation*}
We have $\psi_\lambda'=s_\lambda^2$ and, on $\T$, the change of variables
has Jacobian
\begin{equation*}\label{eq:boundary-jacobian}
dm(\psi_\lambda(\zeta))
=|\psi_\lambda'(\zeta)|\,dm(\zeta)
=|s_\lambda(\zeta)|^2\,dm(\zeta).
\end{equation*}
It follows that $\norm{Uh}_2=\norm{h}_2$ for $h\in L^2(m)$. The inverse
is explicitly
\begin{equation*}\label{eq:unitary-inverse}
(U^{-1}h)(w)=\frac{\sqrt{1-|\lambda|^2}}{1-\bar\lambda w}
 h\!\left(\frac{w-\lambda}{1-\bar\lambda w}\right).
\end{equation*}
Thus $U$ is unitary on $L^2(m)$. Both $U$ and $U^{-1}$ map analytic
polynomials into $H^\infty$. Since analytic polynomials are dense in
$H^2$ and $H^2$ is closed in $L^2$, both operators preserve $H^2$.
They therefore map $H^2$ onto itself. Moreover,
\[
U(\theta h)=s_\lambda (\theta \circ \psi_\lambda) (h \circ \psi_\lambda)= u\,Uh,
\]
so $U(\theta H^2)=uH^2$. Taking orthogonal complements in $H^2$ gives
\begin{equation*}\label{eq:model-unitary}
UK_\theta=K_u,
\qquad UP_\theta U^{-1}=P_u
\end{equation*}
as operators on the appropriate subspaces of $L^2(m)$.

\vspace{11 pt}\noindent 
For a bounded symbol $\varphi$, multiplication operators satisfy
$UM_\varphi U^{-1}=M_{\varphi\circ\psi_\lambda}$. Consequently,
\begin{equation}\label{eq:operator-covariance}
UA_\varphi^\theta U^{-1}
=U P_\theta M_\varphi U^{-1}=P_uUM_\varphi U^{-1}=P_uM_{\varphi\circ\psi_\lambda}UU^{-1}=A_{\varphi\circ\psi_\lambda}^{u}
\quad\text{on }K_u.
\end{equation}
The boundary map $\psi_\lambda$ is a smooth diffeomorphism of $\T$ with
strictly positive Jacobian. It and its inverse preserve null sets.
Composition with it is therefore an isometric bijection of $L^\infty(m)$.
Applying \eqref{eq:operator-covariance} in both directions shows that it
bijects the entire set of bounded symbols of $A$ with that of
$UAU^{-1}$. In particular,
\begin{equation}\label{eq:cost-covariance}
\beta_u(UAU^{-1})=\beta_\theta(A)
\end{equation}
whenever $A$ has a bounded symbol. 

\vspace{11 pt}\noindent 
To identify the transformed rank-one operator, observe that
\[
1-\bar\lambda\psi_\lambda(z)
 =\frac{1-|\lambda|^2}{1+\bar\lambda z},
\qquad
\psi_\lambda(z)-\lambda
 =\frac{(1-|\lambda|^2) z}{1+\bar\lambda z}.
\]
By exact computations, we obtain that
\begin{align*}
    (Uk_\lambda^\theta)(z)
&=\frac{\sqrt{1-|\lambda|^2}}{1+\overline{\lambda}z}
  \frac{1-\overline{\theta(\lambda)}
                 \theta(\psi_\lambda(z))}
       {1-\overline{\lambda}\psi_\lambda(z)}=\frac{\sqrt{1-|\lambda|^2}}{1+\overline{\lambda}z}
  \frac{1-\overline{u(0)}u(z)}
       {(1-|\lambda|^2)/(1+\overline{\lambda}z)}\\
&=\frac{1-\overline{u(0)}u(z)}{\sqrt{1-|\lambda|^2}}
=\frac{k_0^u(z)}{\sqrt{1-|\lambda|^2}}
\end{align*}
 and 
\begin{align*}
(U\widetilde{k}_\lambda^\theta)(z)
&=\frac{\sqrt{1-|\lambda|^2}}{1+\overline{\lambda}z}
  \frac{\theta(\psi_\lambda(z))-\theta(\lambda)}
       {\psi_\lambda(z)-\lambda}=\frac{\sqrt{1-|\lambda|^2}}{1+\overline{\lambda}z}
  \frac{u(z)-u(0)}
       {(1-|\lambda|^2) z/(1+\overline{\lambda}z)}\\
&=\frac{1}{\sqrt{1-|\lambda|^2}}\frac{u(z)-u(0)}{z}=\frac{\widetilde{k}_0^u(z)}{\sqrt{1-|\lambda|^2}}.
\end{align*}
For a unitary $U$, $U(v\otimes w)U^{-1}=Uv\otimes Uw$. Since
\[
\norm{k_0^u}_2^2=1-|u(0)|^2=1-|\theta(\lambda)|^2
=(1-|\lambda|^2)\|k^\theta_\lambda\|^2,
\]
all factors in $Uk_\lambda^\theta$ and $U\widetilde{k}_\lambda^\theta$ cancel in the normalization:
\begin{equation*}\label{eq:rank-one-covariance}
UR_\lambda^\theta U^{-1}=U\left( \frac{\widetilde{k}_\lambda^\theta\otimes k^\theta_\lambda }{\|k^\theta_\lambda\|^2}\right)U^{-1}=\frac{ \widetilde{k}_0^u\otimes k^u_0 }{(1-|\lambda|^2)\|k^\theta_\lambda\|^2}=R_0^u,
\qquad
\beta_u(R_0^u)=\beta_\theta(R_\lambda^\theta).
\end{equation*}
Finally,
\[
u(0)=\theta(\lambda),\qquad
u'(0)=(1-|\lambda|^2)\theta'(\lambda),\qquad
\frac{|u'(0)|}{1-|u(0)|^2}=\eta_\theta(\lambda).
\]
Applying \eqref{eq:center-l2} to $u$ proves \eqref{eq:global-l2}.

\end{proof}

%\vspace{11 pt}
\section{A geometric lemma for holomorphic self-maps}\label{Sect4}

\noindent
The following lemma is the second key step. 
We recall that $A\subset X$ is a \emph{retract} of $X$ if
there exists a continuous map
\[
r:X\longrightarrow A
\]
such that $r(a)=a$ for every $a\in A$.
Such a map $r$ is called a \emph{retraction of $X$ onto $A$}.
\begin{lemma}\label{lem:geometry}
Let $F:\D\to\D$ be a non-constant holomorphic function. Suppose that
there are $r\in(0,1)$ and $c>0$ such that
\begin{equation}\label{eq:hyperbolic-lower}
\frac{(1-|z|^2)|F'(z)|}{1-|F(z)|^2}\ge c
\qquad\text{whenever }|F(z)|>r.
\end{equation}
For every $s\in(r,1)$, the closed sublevel set
$E_s=\{z\in\D:|F(z)|\le s\}$ is nonempty and is a continuous retract
of $\D$. For every $R\in(r,1)$, the open sublevel set
$\Omega_R=\{z\in\D:|F(z)|<R\}$ is nonempty and path connected.
\end{lemma}
\begin{proof}
Fix $r<s<1$ and write $\Omega=\{z\in\D:|F(z)|>r\}$.
On this open set, \eqref{eq:hyperbolic-lower} implies
\[
|F'(z)|\ge c\frac{1-|F(z)|^2}{1-|z|^2}>0.
\]
Thus there are no critical points of $F$ in $\Omega$, and the vector field
\begin{equation*}\label{eq:vector-field}
V(z)=-\frac{F(z)}{F'(z)}
\end{equation*}
is holomorphic there. Viewed as a real vector field on an open subset of
$\mathbb R^2$, it is locally Lipschitz. The local existence and uniqueness
theorem for ordinary differential equations therefore gives a local flow
$\phi(t,z)$ with continuous dependence on $t$ and $z$.

\vspace{11 pt}
\noindent
Let $z_0\in\D$ satisfy $b=|F(z_0)|>s$, and let
$z(t)=\phi(t,z_0)$ be its maximal forward solution in $\Omega$, defined
for $0\le t<\tau$. We have that
\[
\frac{d}{dt}F(z(t))
=F'(z(t))\dot z(t)=-F(z(t)).
\]
Solving this differential equation yields
\begin{equation}\label{eq:flow-image}
F(z(t))=e^{-t}F(z_0),\qquad |F(z(t))|=be^{-t}
\quad(0\le t<\tau).
\end{equation}
The time at which the image should reach modulus $s$ is
\begin{equation*}\label{eq:stopping-time}
T(z_0)=\log\frac{b}{s}>0.
\end{equation*}

\noindent
We next justify existence up to and beyond $T(z_0)$.
Use the hyperbolic metric on $\D$ with length element
$|dz|/(1-|z|^2)$, and denote its distance by $d_\D$. This normalization
gives $d_\D(0,z)=\arctanh|z|$.
Hyperbolic distance dominates Euclidean distance,
and on every closed Euclidean disk compactly contained in \(\mathbb D\)  the two metrics induce the same topology.
Whenever the trajectory exists and $0\le t\le T(z_0)$,
\eqref{eq:hyperbolic-lower} and \eqref{eq:flow-image} imply
\begin{equation}\label{eq:flow-speed}
\frac{|\dot z(t)|}{1-|z(t)|^2}
=\frac{|F(z(t))|}{(1-|z(t)|^2)|F'(z(t))|}
\le\frac1c\frac{be^{-t}}{1-b^2e^{-2t}}.
\end{equation}
The function on the right is defined and continuous on the whole interval
$[0,T(z_0)]$, independently of how long the solution is initially known
to exist. Its denominator is at least $1-b^2>0$.

Suppose for a contradiction that $\tau\le T(z_0)$. For
$0\le t_1<t_2<\tau$, integration along the trajectory gives
\[
d_\D(z(t_1),z(t_2))
\le\frac1c\int_{t_1}^{t_2}
        \frac{be^{-t}}{1-b^2e^{-2t}}\,dt.
\]
The integral tends to zero as $t_1,t_2\to \tau$. Thus the trajectory
is Cauchy in the complete hyperbolic metric and has a limit
$z_\tau\in\D$. By continuity of $F$,
\[
|F(z_\tau)|=be^{-\tau}\ge be^{-T(z_0)}=s>r.
\]
The limit therefore lies in $\Omega$, where $F'\ne0$ and the vector field
is regular. The usual continuation theorem extends the solution past
$\tau$, contradicting maximality. More explicitly, define
$z(\tau)=z_\tau$. The integral equation for $\dot z=V(z)$ extends to
$\tau$ by continuity of $V$, and a local solution starting at $z_\tau$
can be joined to it using uniqueness. We have proved
$\tau>T(z_0)$. In particular, the endpoint at the prescribed time exists
and has $F$-modulus exactly $s$.

\vspace{11 pt}\noindent
Define
\begin{equation}\label{eq:retraction}
\Ret_s(z)=
\begin{cases}
z,& |F(z)|\le s,\\[2pt]
\phi\!\left(\log\dfrac{|F(z)|}{s},z\right),& |F(z)|>s.
\end{cases}
\end{equation}
The preceding argument makes this well defined on all of $\D$.
Every point outside $E_s$ is mapped to its boundary level $|F|=s$, and
every point of $E_s$ is fixed. Hence $\Ret_s(\D)=E_s$.
This set is nonempty: if $|F|\le s$ everywhere it is all of $\D$;
otherwise it contains an endpoint of a trajectory just constructed.

On $\{|F|<s\}$ the map is the identity.
On $\{|F|>s\}$ it is continuous
by joint continuous dependence of the flow and continuity of the stopping
time. To apply that dependence at an endpoint, recall that each relevant
solution exists beyond its stopping time; the domain of the local flow
is open around that time and initial condition.

It remains to check continuity on the level $|F|=s$.
Integration of \eqref{eq:flow-speed} from $0$ to $T(z)$ gives, for
$|F(z)|>s$,
\begin{align*}
d_\D(z,\Ret_s(z))
&\le\frac1c\int_0^{T(z)}
 \frac{|F(z)|e^{-t}}{1-|F(z)|^2e^{-2t}}\,dt=\frac1c\bigl(\arctanh|F(z)|-\arctanh s\bigr).
\label{eq:retraction-distance}
\end{align*}
%The equality follows either by substituting $q=|F(z)|e^{-t}$ or by differentiating $\arctanh(|F(z)|e^{-t})$.
If $[x]_+=\max\{x,0\}$, the estimate valid for all $z\in\D$ is
\begin{equation*}\label{eq:positive-part-distance}
d_\D(z,\Ret_s(z))
\le\frac1c\bigl[\arctanh|F(z)|-\arctanh s\bigr]_+.
\end{equation*}
For $z_n\to p$ with $|F(p)|=s$, its right-hand side tends to zero.
Since Euclidean convergence to an interior point is also hyperbolic
convergence, the triangle inequality yields
\[
d_\D(\Ret_s(z_n),p)
\le d_\D(\Ret_s(z_n),z_n)+d_\D(z_n,p)\longrightarrow0.
\]
This handles sequences approaching the level from either side, and proves
continuity there. Thus $\Ret_s$ is a continuous retraction onto $E_s$.

\vspace{11 pt}\noindent 
The image of a path connected space under a continuous map is path
connected. Here a path in $E_s$ between $p',q'\in E_s$ is explicitly
\begin{equation}\label{eq:closed-sublevel-path}
\gamma(t)=\Ret_s((1-t)p'+tq'),\qquad 0\le t\le1,
\end{equation}
because $\D$ is convex and $\Ret_s$ fixes $p'$ and $q'$.
Now fix $R\in(r,1)$ and choose $s$ with $r<s<R$.
The nonempty set $E_s$ is contained in $\Omega_R$.
If $p\in\Omega_R\setminus E_s$, the path
\[
\alpha_p(t)=\phi(tT(p),p),\qquad 0\le t\le1,
\]
joins $p$ to $\Ret_s(p)$ and stays in $\Omega_R$, since
$|F(\alpha_p(t))|=e^{-tT(p)}|F(p)|\le|F(p)|<R$.
For $p\in E_s$, take the constant path. Given $p,q\in\Omega_R$,
concatenate $\alpha_p$, the path \eqref{eq:closed-sublevel-path} between
$\Ret_s(p)$ and $\Ret_s(q)$, and $\alpha_q$ in reverse. All three pieces
lie in $\Omega_R$, proving the required path connectedness.

\end{proof}

\noindent
The geometric mechanism in Lemma~\ref{lem:geometry} is closely
related to the covering characterization of one-component inner
functions obtained by Ivrii and Urba\'nski.
Their Theorem~7.2 in \cite{ivrii} states, in particular,
that a non-constant inner function $F$ is one-component if and
only if, for some $\rho\in(0,1)$, the restriction
\[
F:F^{-1}(A_\rho)\longrightarrow A_\rho,
\qquad
A_\rho=\{w\in\D:\rho<|w|<1\},
\]
is a covering map.
The lower bound on the hyperbolic derivative in
Lemma~\ref{lem:geometry} provides a quantitative condition
ensuring continuation of inverse branches.

%\vspace{11 pt}
\section{Proof of Theorem \ref{thm:main}}\label{Sect5}

\noindent 
We can now write the proof of Theorem \ref{thm:main}.
\begin{proof}[Proof of Theorem~\ref{thm:main}]

Suppose (i) holds. Lemma~\ref{lem:quotient} and
$\norm{R_\lambda^\theta}=1$ imply
\[
\beta_\theta(R_\lambda^\theta)\le B_\theta<\infty
\qquad(\lambda\in\D).
\]

\vspace{11 pt}\noindent 
Assume (ii). Lemma~\ref{lem:rank-one} gives $M_\theta\ge1$.
Proposition~\ref{prop:covariance} yields
\begin{equation*}\label{eq:eta-from-cost}
\left( \frac{(1-|\lambda|^2)|\theta'(\lambda)|}{1-|\theta(\lambda)|^2}\right)^2=\eta_\theta(\lambda)^2
\ge1-M_\theta^2(1-|\theta(\lambda)|^2).
\end{equation*}
Set
\[
r_M=\sqrt{1-\frac1{2M_\theta^2}}\in(0,1).
\]
Whenever $|\theta(\lambda)|>r_M$, we have
$M_\theta^2(1-|\theta(\lambda)|^2)<1/2$, and hence
\[
\eta_\theta(\lambda)>\frac1{\sqrt2}.
\]
Lemma~\ref{lem:geometry}, with $F=\theta$, $r=r_M$, and
$c=1/\sqrt2$, shows that every sublevel set with $r_M<R<1$ is nonempty
and path connected. This proves (iii).

\vspace{11 pt}\noindent 
Finally, (iii)$\Rightarrow$(i) is the bounded-symbol theorem of
Baranov, Bessonov, and Kapustin, \cite[Corollary 2.5]{bbk}.

\end{proof}

\noindent 
We point out that Lemma \ref{lem:geometry} and Theorem \ref{thm:main} provide a geometric characterization of one-component inner functions. See also \cite[Theorem 7.2]{ivrii}.

\begin{corollary}
Let $\theta$ be a non-constant inner function. Then $\theta$ is one-component if and only if
there are $r\in(0,1)$ and $c>0$ such that
\begin{equation*}\label{eq:hyperbolic-lower2}
\frac{(1-|z|^2)|\theta'(z)|}{1-|\theta(z)|^2}\ge c
\qquad\text{whenever }|\theta(z)|>r.
\end{equation*}
\end{corollary}

\noindent 
The failure of the one-component property has a quantitative
manifestation in the rank-one family considered above. Although
each operator $R_\lambda^\theta$ admits a bounded symbol, the
minimal norms of these symbols cannot remain uniformly bounded.
\begin{corollary}\label{cor:blowup}
If $\theta$ is not one-component, there is a sequence
$(\lambda_n)\subset\D$ such that
\[
|\theta(\lambda_n)|\longrightarrow1,
\qquad \eta_\theta(\lambda_n)=\frac{(1-|\lambda_n|^2)|\theta'(\lambda_n)|}{(1-|\theta(\lambda_n)|^2)}<\frac12,
\]
and
\begin{equation*}\label{eq:cost-blowup}
\beta_\theta(R_{\lambda_n}^\theta)
\ge\frac{\sqrt3}{2\sqrt{1-|\theta(\lambda_n)|^2}}
\longrightarrow\infty.
\end{equation*}
\end{corollary}
\begin{proof}
For every $r\in(0,1)$ there must be a point $\lambda$ with
$|\theta(\lambda)|>r$ and $\eta_\theta(\lambda)<1/2$.
Otherwise Lemma~\ref{lem:geometry}, applied with $c=1/2$, would make
$\theta$ one-component. Choose such points for a sequence $r_n\uparrow1$.
Estimate \eqref{eq:global-l2} then gives
\[
\beta_\theta(R_{\lambda_n}^\theta)^2
\ge\frac{1-\eta_\theta(\lambda_n)^2}
         {1-|\theta(\lambda_n)|^2}
\ge\frac{3}{4(1-|\theta(\lambda_n)|^2)},
\]
which proves the assertion.

\end{proof}

\noindent The preceding divergence can be converted into the existence of
a single operator without a bounded symbol. Since the normalized
rank-one operators have trace norm one, the closed graph theorem
allows this obstruction to be realized within the trace class.

\begin{corollary}\label{cor:trace-class}
If $\theta$ is not one-component, there exists a trace-class truncated
Toeplitz operator on $K_\theta$ with no bounded symbol.
\end{corollary}
\begin{proof}
Let $S_1(K_\theta)$ denote the trace class, with norm $\norm{\cdot}_{S_1}$,
and set
\[
X=\TT(\theta)\cap S_1(K_\theta).
\]
Since $\norm{A}\le\norm{A}_{S_1}$ and $\TT(\theta)$ is closed in
operator norm, $X$ is closed in $S_1(K_\theta)$. It is consequently
Banach for the trace norm.
As in the proof of Lemma~\ref{lem:quotient}, let
$E_\theta=L^\infty(m)/\ker Q_\theta$. The induced map
\[
\widetilde Q_\theta:E_\theta\longrightarrow\TT(\theta)
\]
is continuous and injective, regardless of whether it is surjective.

\vspace{11 pt}\noindent 
Suppose by contradiction that every element of $X$ has a bounded symbol. Then
\[
L:X\longrightarrow E_\theta,\qquad
L(A)=\widetilde Q_\theta^{-1}A
\]
is a well-defined linear map on all of $X$. Its graph is closed.
Indeed, if $A_n\to A$ in trace norm and $L(A_n)\to e$ in $E_\theta$,
continuity of $\widetilde Q_\theta$ gives
\[
\widetilde Q_\theta e
=\lim_n\widetilde Q_\theta L(A_n)=\lim_n A_n=A
\]
in operator norm. Injectivity yields $e=L(A)$.
The closed graph theorem therefore gives a constant $C<\infty$ with
\[
\beta_\theta(A)=\norm{L(A)}_{E_\theta}
\le C\norm{A}_{S_1}\qquad(A\in X).
\]
By Lemma~\ref{lem:rank-one}, each $R_\lambda^\theta$ belongs to $X$ and
has trace norm one. Hence $M_\theta\le C$, and
Theorem~\ref{thm:main} implies that $\theta$ is one-component, a
contradiction.

\end{proof}
\noindent 
%Corollary \ref{cor:trace-class} can be clearly reformulated also for Schatten class $p>1$ in $\TT(\theta)$.

%\vspace{11 pt}
\section{Consequences of Theorem \ref{thm:main}}\label{Sect6}

\noindent 
Theorem~\ref{thm:main} supplies the geometric condition missing from several previous characterizations of the bounded symbol property. In this section we combine it with the results of Baranov, Bessonov, and
Kapustin~\cite{bbk}, O'Loughlin~\cite{O'Loughlin}, and Aleksandrov~\cite{Ale2}.

\subsection{Carleson embeddings and the exact factorization space}
For $1\le p<\infty$, write
\[
K_\omega^p=H^p\cap\omega\overline{zH^p}
\]
for the model space associated with an inner function $\omega$; thus
$K_\omega^2=K_\omega$. Following \cite{bbk}, let $\mathcal C_p(\omega)$ denote
the class of finite complex Borel measures $\mu$ on $\T$ for which
\[
K_\omega^p\hookrightarrow L^p(|\mu|)
\]
is bounded, where $|\mu|$ is the total variation measure. The analogous
class for measures on $\overline\D$ is denoted by
$\mathcal D_p(\omega)$. Boundary embeddings are initially defined on the
dense subspace of model-space functions continuous on $\overline\D$ and
then extended. 

\vspace{11 pt}\noindent 
The analytic space relevant to factorization is
\begin{equation*}\label{eq:factorization-target}
\mathcal B_\theta
=\{f\in H^1:\theta^2\bar z^{\,2}\overline f\in H^1\}
=z^{-1}\bigl(K_{\theta^2}^1\cap zH^1\bigr).
\end{equation*}
The space $\mathcal B_\theta$ is closed in $H^1$, since multiplication
by $\theta^2\bar z^{\,2}$ followed by conjugation is continuous on
$L^1(m)$. If $\theta(0)=0$, then $\theta^2/z$ is inner and
\begin{equation*}\label{eq:normalized-factorization-target}
\mathcal B_\theta=K_{\theta^2/z}^1.
\end{equation*}

\vspace{11 pt}\noindent 
Define the analytic weak product
\begin{equation}\label{eq:weak-product}
X_\theta^a=
\left\{\sum_{j=1}^\infty x_jy_j:
 x_j,y_j\in K_\theta,\quad
 \sum_{j=1}^\infty\norm{x_j}_2\norm{y_j}_2<\infty\right\}.
\end{equation}
Every such series converges in $H^1$, by H\"older's inequality. Moreover,
\[
\theta^2\bar z^{\,2}\overline{x_jy_j}
=(C_\theta x_j)(C_\theta y_j)\in H^1,
\]
so $X_\theta^a\subseteq\mathcal B_\theta$. Equip the weak product with
its projective norm
\begin{equation*}\label{eq:projective-norm}
\pi_\theta(f)=\inf\left\{
 \sum_j\norm{x_j}_2\norm{y_j}_2: f=\sum_jx_jy_j\right\}.
\end{equation*}
This is the quotient norm that makes $X_\theta^a$ a Banach space.
H\"older's inequality gives $\norm{f}_1\le\pi_\theta(f)$.

\begin{corollary}\label{cor:embedding-factorization}
For a non-constant inner function $\theta$, the following are equivalent:
\begin{enumerate}[label=\textup{(\roman*)}]
\item $\theta$ is one-component.
\item $\mathcal C_1(\theta^2)=\mathcal C_2(\theta^2)$.
\item $\mathcal D_1(\theta^2)=\mathcal D_2(\theta^2)$.
\item $X_\theta^a=\mathcal B_\theta$ as sets.
\item There is $F_\theta<\infty$ such that every $f\in\mathcal B_\theta$
 admits a representation
\begin{equation}\label{eq:four-factorization}
f=\sum_{j=1}^4 x_jy_j,
\qquad x_j,y_j\in K_\theta,
\qquad
\sum_{j=1}^4\norm{x_j}_2\norm{y_j}_2
\le F_\theta\norm{f}_1.
\end{equation}
\end{enumerate}
The measure-class equalities can equivalently be imposed on positive
measures only. 
\end{corollary}
\noindent 
For every inner function, independently of these equivalent
conditions,
\begin{equation}\label{eq:unconditional-embedding-inclusions}
\begin{split}
\mathcal C_1(\theta^2)&\subseteq\mathcal C_2(\theta^2)
 =\mathcal C_2(\theta),\\
\mathcal D_1(\theta^2)&\subseteq\mathcal D_2(\theta^2)
 =\mathcal D_2(\theta).
\end{split}
\end{equation}
See \cite[Lemma 6.1]{bbk}. If $\theta$ is not one-component, both inclusions are strict and some
$f\in\mathcal B_\theta$ has no representation of the form
\eqref{eq:weak-product}. See also \cite{Bell-Pel,kul} for the weak factorization property and its connections with duality. 
\begin{proof}
Theorem 2.4 in \cite{bbk} and the discussion immediately following it identify
(ii), (iii), and (iv) with the bounded symbol property. We apply
Theorem~\ref{thm:main} to identify that property with (i).
When (iv) holds, the identity map from the Banach space $X_\theta^a$
onto $\mathcal B_\theta$ is continuous and bijective. Its inverse is
bounded, so $\pi_\theta(f)\le C_\theta\norm{f}_1$.
The reduction to four products, with control of their norms, is
\cite[Proposition 4.1]{bbk}; this gives (v). Conversely, (v) directly
implies (iv).
Finally,
membership of a complex measure in an embedding class depends only on
its total variation, which proves the assertion about positive measures.

\end{proof}

%\vspace{11 pt}
\subsection{Compact operators and O'Loughlin's conjecture}
For a nonconstant inner function $\theta$, define
\[
\mathcal S_\theta
=\{A_{\theta g}^\theta:g\in C(\T)\}.
\]
The factor $\theta$ in this definition is essential: a symbol in
$\theta C(\T)$ need not itself be continuous on $\T$.
O'Loughlin proves that the bounded symbol property is equivalent to
$\mathcal S_\theta=\TT_0(\theta)$, the space of compact truncated Toeplitz operators \cite[Theorem 3.2]{O'Loughlin}. His Conjecture 3.2 asks whether this
compact-symbol property characterizes one-component inner functions.

\begin{corollary}\label{cor:compact-symbols}
For a nonconstant inner function $\theta$, the following are equivalent:
\begin{enumerate}[label=\textup{(\roman*)}]
\item $\theta$ is one-component.
\item Every compact truncated Toeplitz operator on $K_\theta$ admits a
 symbol of the form $\theta g$, with $g\in C(\T)$.
\item $\mathcal S_\theta$ is closed in the operator norm.
\end{enumerate}
For every $\theta$, $\mathcal S_\theta$ is dense in $\TT_0(\theta)$;
it is therefore a proper dense subspace when $\theta$ is not
one-component. If $\theta$ is one-component, there is a constant
$D_\theta<\infty$ such that
\begin{equation}\label{eq:continuous-symbol-cost}
\norm{A}\le\beta_\theta(A)\le\gamma_\theta(A)
\le D_\theta\norm{A},\qquad A\in\TT_0(\theta),
\end{equation}
where
\[
\gamma_\theta(A)
=\inf\{\norm{g}_\infty:g\in C(\T),\ A=A_{\theta g}^\theta\}.
\]
\end{corollary}
\begin{proof}
First, $A_{\theta g}^\theta$ is compact for every continuous $g$, without
any one-component assumption~\cite[Proposition 5.4]{Chal-fri}. 
The equivalence (i)$\Leftrightarrow$(ii) follows by combining
Theorem~\ref{thm:main} with O'Loughlin's Theorem 3.2. The density
statement is \cite[Corollary 3.9]{O'Loughlin}. Since
$\TT_0(\theta)$ is closed in operator norm, density makes (ii) and
(iii) equivalent.
Under (ii), the map
\[
C(\T)\longrightarrow\TT_0(\theta),\qquad
g\longmapsto A_{\theta g}^\theta
\]
is a continuous linear surjection between Banach spaces. Its quotient
norm is $\gamma_\theta$, so the open mapping theorem gives the last
inequality in \eqref{eq:continuous-symbol-cost}. The other inequalities
follow from $|\theta|=1$ almost everywhere and the definitions.

\end{proof}

\subsection{Further estimates for $B_\theta$}

It follows directly from the definitions that $M_\theta$, as defined in
Theorem~\ref{thm:main}, is controlled by $B_\theta$. It is natural to ask
whether the reverse inequality also holds when $\theta$ is one-component and,
if so, whether it holds with a constant independent of $\theta$.

\vspace{11pt}\noindent
We obtain only an answer to this question in the particular case where $\theta$
is a Blaschke product of degree $2$.

\begin{theorem}
If $\theta$ is a Blaschke product of degree two, then $B_\theta\le 2 M_\theta$ and $2$ is optimal. 
\end{theorem}
\begin{proof}
Let $\deg \theta=2$ and $A\in \mathcal T(\theta)$ with $\|A\|\le 1$. According to~\cite{sarason}, every operator in $\mathcal T(\theta)$ is $C_\theta$-symmetric and $\dim\mathcal T(\theta)=3$, so in a $C_\theta$-real orthonormal basis, $\mathcal T(\theta)$ is isometrically isomorphic with the space of all complex symmetric $2\times2$ matrices $S$.

For such an $S$ with $\|S\|\le 1$, by the Takagi factorization, there is a unitary $W$ such that
$$
    S=W\begin{pmatrix}
s_1 & 0 \\
0 & s_2 \\
\end{pmatrix} W^T = \sum_{k=1}^2 s_kw_k w_k^T,
$$
where $s_k$ are the singular values of $S$ and $w_k$ is the $k$th column of $W$. Notice that each $w_jw_j^T$ is a rank-one TTO of norm one, so \cite[Theorem~5.1]{sarason} implies that is a unimodular multiple of one of the following:
$$
    R_\lambda^\theta, \quad (R_\lambda^\theta)^*, \quad
    R^\theta_\zeta=\widetilde k^\theta_\zeta\otimes k^\theta_\zeta/\|k^\theta_\zeta\|_2^2
$$
for some $\lambda\in\D$ and $\zeta\in\T$. In the first two cases, since $(A_\varphi^\theta)^*=A_{\bar \varphi}^\theta$, we have
$$
    \beta_\theta((R_\lambda^\theta)^*) = \beta_\theta(R_\lambda^\theta)
    \le M_\theta.
$$
In the third case, using subadditivity of $\beta_\theta$ and the norm convergence $R^\theta_{r\zeta}\to R^\theta_\zeta$ as $r\to 1$, we get
$$
    \beta_\theta(R^\theta_\zeta)
    \le \beta_\theta(R^\theta_{r\zeta})+C\|R^\theta_\zeta-R^\theta_{r\zeta}\|
    \le M_\theta+B_\theta\|R^\theta_\zeta-R^\theta_{r\zeta}\|\to M_\theta.
$$
Therefore,
$$
    \beta_\theta(S)
    \le s_1\beta_\theta(w_1w_1^T) +s_2\beta_\theta(w_2w_2^T)
    \le (s_1+s_2)M_\theta
    \le 2M_\theta\|S\|.
$$

\vspace{11 pt}
\noindent 
To prove optimality, for $0<c<1$, let
$$
    \theta_c(z)=\frac{z^2-c}{1-cz^2}
$$
and write $\beta_c=\beta_{\theta_c}$, $B_c=B_{\theta_c}$, and $M_c=M_{\theta_c}$. Since $\theta_c(0)=-c$ and $\theta_c'(0)=0$, the normalized kernels
$$
    e_0(z)=\frac{k_0^{\theta_c}(z)}{\sqrt{1-c^2}}
    =\frac{\sqrt{1-c^2}}{1-cz^2}
    \quad\text{and}\quad
    e_1(z)=\frac{\widetilde k_0^{\theta_c}(z)}{\sqrt{1-c^2}}=ze_0(z)
$$
are orthogonal and hence form an orthonormal basis of $K_{\theta_c}$.
Since $|e_0|=|e_1|$ on $\T$, each operator in
$\mathcal T(\theta_c)$ has equal diagonal entries in this
basis, and so, since $\dim\mathcal T(\theta_c)=3$, the matrices of $\mathcal T(\theta_c)$ coincide with the $2\times2$ Toeplitz matrices.

\vspace{11 pt}
\noindent 
Let
$$
    X=\begin{pmatrix}0&1\\1&0\end{pmatrix}
    \quad\text{and}\quad
    Y=\begin{pmatrix}0&-1\\1&0\end{pmatrix}.
$$
The identities $ze_0=e_1$ and
$ze_1=ce_0+\sqrt{1-c^2}\,\theta_c$ show that $X$ has
symbol $(z+\bar z)/(1+c)$. It follows that $\beta_c(I)=1$ and $\beta_c(X)\le\frac{2}{1+c}\le2$, and also $\beta_c(Y)\le B_c$ because $\|Y\|=1$.
Every rank-one $2\times 2$ Toeplitz matrix $T$ of norm one can be
written as
$$
    T=\begin{pmatrix}uv&u^2\\v^2&uv\end{pmatrix}
     =uvI+\frac{u^2+v^2}{2}X+\frac{v^2-u^2}{2}Y,
$$
where $|u|^2+|v|^2=1$, and hence
$$
    \beta_c(T)
    \le\frac12+\frac12\beta_c(X)+\frac12\beta_c(Y)
    \le\frac{B_c+3}{2}.
$$
Applying this to $R_\lambda^{\theta_c}$ and taking the sup over $\lambda\in\D$ shows that $2M_c-3\le B_c$, and so
$$
    2-\frac3{M_c}\le\frac{B_c}{M_c}\le2.
$$
It remains to note that, as $c\to 1$,
$$
    M_c\ge\beta_c(R_0^{\theta_c})
    \ge\frac1{\sqrt{1-c^2}}\to\infty,
$$
which follows from Proposition~\ref{prop:center} using $\theta_c(0)$ and $\theta'_c(0)$, which were computed above.

\end{proof}

%\vspace{11 pt}
\section*{Methodology}
The use of AI tools for exploratory discussions and language refinement during the preparation of this work is acknowledged. More specifically, suggestions generated by ChatGPT-6 contributed to the formulation and organization of the argument that appears in Section 4 of this paper. All mathematical statements, proofs, and references were independently checked,
revised, and finalized by the authors, who take full responsibility for the content of the paper.

%\section*{Acknowledgements}

\vspace{22 pt}
\bibliographystyle{habbrv}
\bibliography{biblio}

\end{document}